\pdfoutput=1  
\documentclass[10pt]{amsart}
\usepackage{amscd}
\usepackage{amssymb}
\usepackage[all]{xy}
\usepackage{lmodern}
\usepackage[T1]{fontenc}
\usepackage{graphicx}
\title[Completion and Torsion]{Errata: On the Homology of Completion and 
Torsion}
\date{19 June 2015}
\usepackage{hyperref}
\hypersetup{colorlinks=false}

\author{Marco Porta, Liran Shaul and Amnon Yekutieli}

\address{Porta: Department of  Mathematics
Ben Gurion University,
Be'er Sheva 84105,
Israel}
\email{marcoporta1@libero.it}

\address{Shaul: 
Universiteit Antwerpen, Departement Wiskunde-Informatica, Middelheim \lb 
Campus, Middelheimlaan 1, 2020 Antwerp, Belgium}
\email{Liran.Shaul@uantwerpen.be}

\address{Yekutieli: Department of  Mathematics
Ben Gurion University,
Be'er Sheva 84105,
Israel}
\email{amyekut@math.bgu.ac.il}

\newtheorem{thm}[equation]{Theorem}

\newtheorem{prop}[equation]{Proposition}
\newtheorem{lem}[equation]{Lemma}
\theoremstyle{definition}

\newtheorem{exa}[equation]{Example}

\newcommand{\iso}{\xrightarrow{\simeq}}

\newcommand{\xar}{\xrightarrow}
\newcommand{\opn}{\operatorname}
\newcommand{\cat}[1]{\operatorname{\mathsf{#1}}}

\newcommand{\mfrak}[1]{\mathfrak{#1}}

\newcommand{\mrm}[1]{\mathrm{#1}}
\newcommand{\mbb}[1]{\mathbb{#1}}

\newcommand{\Ga}{\Gamma}
\newcommand{\La}{\Lambda}
\newcommand{\si}{\sigma}

\renewcommand{\a}{\mfrak{a}}

\newcommand{\ba}{\bsym{a}}

\newcommand{\tup}[1]{\textup{#1}}
\newcommand{\bsym}[1]{\boldsymbol{#1}}

\newcommand{\ot}{\otimes}

\newcommand{\lb}{\linebreak}

\begin{document}

\maketitle

In this note we correct two errors in our paper \cite{PSY}. 

\begin{enumerate}
\item There is an error in the proof of Theorem 7.12 (GM Duality) of 
\cite{PSY}. The statement itself is correct, and it is repeated as Theorem 
\ref{thm:1} below, with a correct proof. 

\item Just before formula (3.12) in \cite{PSY}, we said that 
``$\cat{Mod}_{\a\tup{-tor}} A$ is a thick abelian subcategory of  $\cat{Mod} 
A$''. This is true if the ideal $\a$ is finitely generated, but might be false 
otherwise. There is no implication of this error on the rest of the paper, 
since WPR ideals are by definition finitely generated. We thank R. Vyas for 
mentioning this error to us. 
\end{enumerate}

In the rest of the note we recall the definitions and results from \cite{PSY}, 
that are needed to prove Theorem \ref{thm:1}.

Let $A$ be a commutative ring, and let $\a$ be a weakly proregular 
ideal in it (see \cite[Definition 4.21]{PSY}). We choose a finite sequence 
$\ba = (a_1, \ldots, a_n)$ that generates the ideal $\a$. 
The telescope complex associated to this sequence is 
$\opn{Tel}(A; \ba)$. 
There is a canonical homomorphism of complexes
\[ u_{\ba} : \opn{Tel}(A; \ba) \to A . \]
See \cite[Section 5]{PSY}. Let us write $T := \opn{Tel}(A; \ba)$
and $u := u_{\ba}$. For any complex of $A$-modules $M$, we identify $M$ with 
$A \ot_A M$ and with $\opn{Hom}_A(A, M)$. This allows us to form the 
homomorphisms
\[ u \ot \opn{id} : T \ot_A M \to M \]
and 
\[ \opn{Hom}(u, \opn{id}) : M  \to \opn{Hom}_A(T, M) . \]

The mistake in the proof of \cite[Theorem 7.12]{PSY} was as follows. We had 
claimed that for any K-projective complex $P$, the homomorphism of complexes 
\begin{equation} \label{eqn:1}
\opn{Hom}(u, \opn{id}) : T \ot_A P \to 
\opn{Hom}_A(T, T \ot_A P) 
\end{equation}
is a quasi-isomorphism. But this is {\em false}. 
Here is a counterexample: 

\begin{exa}
Take $A := \mbb{K}[[t]]$, the ring of formal power series over a field 
$\mbb{K}$, with $\a := (t)$. Consider the complex $P := A$. Then there are 
derived category isomorphisms
\[ T \ot_A P \cong T \cong \mrm{R} \Ga_{\a}(A) \cong 
\mrm{H}^1(\mrm{R} \Ga_{\a}(A))[-1] \]
and 
\[ \opn{Hom}_A(T, T \ot_A P) \cong 
\opn{Hom}_A(T, T) \cong \mrm{L} \La_{\a}(A) \cong \La_{\a}(A) \cong A . \] 
\end{exa}

The next two results from \cite{PSY} will be crucial for the proof of Theorem 
\ref{thm:1}, so we copy them. 

\begin{prop}[{\cite[Proposition 5.8]{PSY}}] \label{prop:20}
Let $A$ be a commutative ring, let $\a$ be a weakly proregular 
ideal in $A$, and let $\ba$ be a finite sequence  that generates the ideal 
$\a$. For any $M \in \cat{D}(\cat{Mod} A)$ there is an isomorphism 
\[ v^{\mrm{R}}_{\ba, M} : \mrm{R} \Ga_{\a}(M) \iso \opn{Tel}(A; \ba) \ot_A M \]
in $\cat{D}(\cat{Mod} A)$. The isomorphism 
$v^{\mrm{R}}_{\ba, M}$
is functorial in $M$, and the diagram 

\[ \UseTips \xymatrix @C=9ex @R=7ex {
\mrm{R} \Ga_{\a}(M)
\ar[r]^(0.4){ v^{\mrm{R}}_{\ba, M} }
\ar[dr]_{ \si^{\mrm{R}}_{M} }
&
\opn{Tel}(A; \ba) \ot_A M
\ar[d]^{ u_{\ba} \ot \opn{id} }
\\
&
M
} \]

\noindent
in $\cat{D}(\cat{Mod} A)$ is commutative. 
\end{prop}

\begin{prop}[{\cite[Corollary 5.25]{PSY}}] \label{prop:21}
Let $A$ be a commutative ring, let $\a$ be a weakly proregular 
ideal in $A$, and let $\ba$ be a finite sequence  that generates the ideal 
$\a$. For any $M \in \cat{D}(\cat{Mod} A)$ there is an isomorphism 
\[ \opn{tel}^{\mrm{L}}_{\ba, M} : 
\opn{Hom}_A \bigl( \opn{Tel}(A; \ba), M \bigr) \to \mrm{L} \La_{\a}(M) \]
in $\cat{D}(\cat{Mod} A)$. The isomorphism 
$\opn{tel}^{\mrm{L}}_{\ba, M}$
is functorial in $M$, and the diagram 

\[ \UseTips \xymatrix @C=9ex @R=7ex {
M
\ar[d]_{ \opn{Hom}(u_{\ba}, \opn{id}) }
\ar[dr]^{ \tau^{\mrm{L}}_{M} }
\\
\opn{Hom}_A \bigl( \opn{Tel}(A; \ba), M \bigr)
\ar[r]_(0.64){ \opn{tel}^{\mrm{L}}_{\ba, M} }
&
\mrm{L} \La_{\a}(M) 
} \]

\noindent
in $\cat{D}(\cat{Mod} A)$ is commutative. 
\end{prop}

Here are four lemmas that will be needed for the proof of Theorem \ref{thm:1}.
We retain the assumptions of the propositions above, and the shorthand 
$T := \opn{Tel}(A; \ba)$ and $u := u_{\ba}$. The next lemma replaces the  
problematic (\ref{eqn:1}). 

\begin{lem} \label{lem:1}
For any complex of $A$-modules $M$, the homomorphism of complexes
\[ \opn{Hom}_A(\opn{id}, u \ot \opn{id}) :
\opn{Hom}_A(T, T \ot_A M) \to \opn{Hom}_A(T, M) \]
is a quasi-isomorphism. 
\end{lem}

\begin{proof}
We will prove that this is an isomorphism in $\cat{D}(\cat{Mod} A)$. 
Using Propositions \ref{prop:20} and \ref{prop:21}, we may replace the given 
morphism with 
\[ \mrm{L} \La_{\a}(\si^{\mrm{R}}_M) : 
\mrm{L} \La_{\a}(\mrm{R} \Ga_{\a}(M)) \to \mrm{L} \La_{\a}(M) . \]
According to \cite[Lemma 7.2]{PSY}, the morphism 
$\mrm{L} \La_{\a}(\si^{\mrm{R}}_M)$
is an isomorphism.
\end{proof}

The following lemma is actually stated correctly in the proof of 
\cite[Theorem 7.12]{PSY}, but we repeat it for completeness. 

\begin{lem} \label{lem:2}
For any complex of $A$-modules $M$, the homomorphism of complexes
\[ \opn{id} \ot \opn{Hom}(u, \opn{id}) :
T \ot_A M \to T \ot_A \opn{Hom}_A(T, M) \]
is a quasi-isomorphism. 
\end{lem}

\begin{proof}
As in the proof of Lemma \ref{lem:1}, it suffices to prove that the 
morphism
\[ \mrm{R} \Ga_{\a}(\tau^{\mrm{L}}_M) : 
\mrm{R} \Ga_{\a}(M) \to 
\mrm{R} \Ga_{\a}(\mrm{L} \La_{\a}(M))   \]
in $\cat{D}(\cat{Mod} A)$ is an isomorphism. 
This is true by \cite[Lemma 7.6]{PSY}.
\end{proof}

\begin{lem} \label{lem:4}
For any complex of $A$-modules $N$, the homomorphism of complexes
\[ \opn{Hom}(\opn{id}_T, \opn{Hom}(u, \opn{id}_N)) :
\opn{Hom}_A (T, N) \to \opn{Hom}_A(T, \opn{Hom}_A(T, N)) \]
is a quasi-isomorphism. 
\end{lem}

\begin{proof}
Using Hom-tensor adjunction, we can replace the given homomorphism 
with the homomorphism
\[ \opn{Hom}(u \ot \opn{id}, \opn{id}) :
\opn{Hom}_A (T, N) \to \opn{Hom}_A( T \ot_A T, N)  \]
in $\cat{C}(\cat{Mod} A)$, that is isomorphic to it. 
Now, by \cite[Corollary 7.9]{PSY}, the homomorphism 
\[ u \ot \opn{id}_T : T \ot_A T \to T \]
is a homotopy equivalence. Hence $\opn{Hom}(u \ot \opn{id}_T, \opn{id}_N)$
is a quasi-isomorphism.
\end{proof}

\begin{lem} \label{lem:20} 
For any complex of $A$-modules $N$, the homomorphism of complexes
\[ u \ot \opn{id}_T \ot \opn{id}_N :  T \ot_A T \ot_A N 
\to T \ot_A N \]
is a quasi-isomorphism. 
\end{lem}

\begin{proof}
We already know that $u \ot \opn{id}_T$ is a homotopy equivalence.
\end{proof}

\begin{thm}[GM Duality] \label{thm:1}
Let $A$ be a commutative ring, and let $\a$ be a weakly proregular ideal in 
$A$. For any $M, N \in \cat{D}(\cat{Mod} A)$ the morphisms 
\[ \begin{aligned}
\opn{RHom}_A \bigl( \mrm{R} \Gamma_{\a} (M), \mrm{R} \Gamma_{\a} (N) \bigr)
& \xar{\opn{RHom}(\opn{id}, \sigma^{\mrm{R}}_N)}
\opn{RHom}_A \bigl( \mrm{R} \Gamma_{\a} (M), N \bigr) 
\\
& \xar{\opn{RHom}(\opn{id}, \tau^{\mrm{L}}_N)}
\opn{RHom}_A \bigl( \mrm{R} \Gamma_{\a} (M), \mrm{L} \Lambda_{\a} (N) \bigr)
\\
&  
\xleftarrow{\opn{RHom}(\sigma^{\mrm{R}}_M, \opn{id})}
\opn{RHom}_A \bigl( M, \mrm{L} \Lambda_{\a} (N) \bigr)
\\
& 
\xleftarrow{\opn{RHom}(\tau^{\mrm{L}}_M, \opn{id})}
\opn{RHom}_A \bigl( \mrm{L} \Lambda_{\a} (M), \mrm{L} \Lambda_{\a} (N) \bigr)
\end{aligned} \]
in $\cat{D}(\cat{Mod} A)$ are isomorphisms.
\end{thm}

\begin{proof}
Choose a weakly proregular sequence $\bsym{a}$ that generates $\a$, and write 
$T := \opn{Tel}(A; \bsym{a})$ and $u := u_{\bsym{a}}$.
Next choose a K-projective resolution $P \to M$, and a K-injective resolution 
$N \to I$. The complex $T \ot_A P$ is K-projective, and the complex 
$\opn{Hom}_A(T, I)$ is K-injective.  

By Propositions \ref{prop:20} and \ref{prop:21}, we can replace 
the diagram above with the diagram 
\begin{equation} \label{eqn:25}
\begin{aligned}
& \opn{Hom}_A \bigl( T \ot_A P, T \ot_A I \bigr)
\\
& \qquad \xar{\opn{Hom}(\opn{id} \ot \opn{id}, u \ot \opn{id} )}
\opn{Hom}_A \bigl( T \ot_A P, I \bigr) 
\\
& \qquad \xar{\opn{Hom}(\opn{id} \ot \opn{id}, \opn{Hom}( u, \opn{id}))}
\opn{Hom}_A \bigl( T \ot_A P, \opn{Hom}_A(T, I) \bigr)
\\ 
& \qquad
\xleftarrow{ \opn{Hom}( u \ot \opn{id}, \opn{Hom}( \opn{id}, \opn{id})) }
\opn{Hom}_A \bigl( P, \opn{Hom}_A(T, I) \bigr)
\\ 
& \qquad
\xleftarrow{ \opn{Hom}(\opn{Hom}(\opn{id}, u), \opn{Hom}( \opn{id}, \opn{id})) }
\opn{Hom}_A \bigl( \opn{Hom}_A(T, P), \opn{Hom}_A(T, I) \bigr)
\end{aligned}
\end{equation}
in $\cat{C}(\cat{Mod} A)$. We will prove that all these homomorphisms are
quasi-isomorphisms. 

First we deal with the two forward facing homomorphisms in (\ref{eqn:25}). 
By Hom-tensor adjunction we can move the complex $P$ to the left in all 
expressions, thus obtaining an isomorphic diagram 
\[ \begin{aligned}
& \opn{Hom}_A \bigl( P, \opn{Hom}_A (T, T \ot_A I) \bigr)
\\
& \qquad \xar{\opn{Hom}(\opn{id}, \opn{Hom}(\opn{id}, u \ot \opn{id}))}
\opn{Hom}_A \bigl( P, \opn{Hom}_A(T, I) \bigr) 
\\
& \qquad \xar{\opn{Hom}(\opn{id}, \opn{Hom}(\opn{id}, \opn{Hom}( u, \opn{id})))}
\opn{Hom}_A \bigl( P, \opn{Hom}_A(T, \opn{Hom}_A(T, I)) \bigr) 
\end{aligned} \]
in $\cat{C}(\cat{Mod} A)$.
Because $P$ is K-projective, it suffices to prove that 
\[ \begin{aligned}
& \opn{Hom}_A (T, T \ot_A I) 
\\
& \qquad \xar{ \opn{Hom}(\opn{id}, u \ot \opn{id}) }
\opn{Hom}_A(T, I) 
\\
& \qquad \xar{ \opn{Hom}(\opn{id}, \opn{Hom}( u, \opn{id})) }
\opn{Hom}_A(T, \opn{Hom}_A(T, I))
\end{aligned} \]
are quasi-isomorphisms. This is true by Lemmas \ref{lem:1} and \ref{lem:4}.

Now we deal with the backward facing homomorphisms in (\ref{eqn:25}). 
Using Hom-tensor adjunction, we can move the rightmost occurrence of the 
complex $T$ to the left, thus obtaining an isomorphic diagram 
\[ \begin{aligned}
& 
\opn{Hom}_A ( T \ot_A T \ot_A P, I )
\\ 
& \qquad
\xleftarrow{ \opn{Hom}( u \ot \opn{id} \ot \opn{id}, \opn{id}) }
\opn{Hom}_A (T \ot_A P, I) 
\\ 
& \qquad
\xleftarrow{ \opn{Hom}( \opn{id} \ot  \opn{Hom}(\opn{id}, u), \opn{id}) }
\opn{Hom}_A \bigl( T \ot_A \opn{Hom}_A(T, P),  I \bigr) 
\end{aligned} \]
in $\cat{C}(\cat{Mod} A)$.
To prove these are quasi-isomorphisms, we can remove the K-injective complex 
$I$. Thus it suffices to show that 
\[ \begin{aligned}
& 
T \ot_A T \ot_A P 
\\ 
& \qquad
\xrightarrow{ u \ot \opn{id} \ot \opn{id} }
T \ot_A P 
\\ 
& \qquad
\xrightarrow{ \opn{id} \ot  \opn{Hom}(u, \opn{id}) }
T \ot_A \opn{Hom}_A(T, P) 
\end{aligned} \]
are quasi-isomorphisms. This is true by Lemmas \ref{lem:2} and \ref{lem:20}.
\end{proof}

We should remark that there was a correct proof of a weaker version of Theorem 
\ref{thm:1} in earlier versions of the paper \cite{PSY} (the assumption was 
that $A$ is noetherian). This weaker statement had already been proved in 
\cite{AJL}.



\begin{thebibliography}{SGA~4}

\bibitem[AJL]{AJL} L. Alonso, A. Jeremias and J. Lipman, 
  Local homology and cohomology on schemes, 
  Ann.\ Sci.\ ENS {\bf 30} (1997), 1-39.
  Correction, availabe online at 
\url{http://www.math.purdue.edu/~lipman/papers/homologyfix.pdf}. 

\bibitem[PSY]{PSY} M. Porta, L. Shaul and A. Yekutieli,
On the Homology of Completion and Torsion,
Algebras and Repesentation Theory (2014) 17:31-67.


\end{thebibliography}
\end{document}